\documentclass{amsart}

\usepackage{amsmath,amssymb,mathtools}
\usepackage{booktabs}
\usepackage{enumitem}
\usepackage{centernot}
\usepackage[hidelinks]{hyperref}

\newcommand{\Z}{\mathbb Z}
\newcommand{\Aff}{\overrightarrow{\mathrm{GL}}}
\newcommand{\sign}{\operatorname{sign}}

\theoremstyle{plain}
\newtheorem{theorem}{Theorem}[section]
\newtheorem{proposition}[theorem]{Proposition}
\newtheorem{corollary}[theorem]{Corollary}
\newtheorem{lemma}[theorem]{Lemma}
\theoremstyle{remark}
\newtheorem{remark}[theorem]{Remark}

\title{The Wang--Sun Sum over Affine Derangements}
\author{Octavio A. Agust\'in-Aquino}
\address{Instituto de Física y Matemáticas\\Universidad Tecnológica de la Mixteca\\Huajuapan de León, Oaxaca, México}
\email{octavioalberto@mixteco.utm.mx}
\date{September 29, 2026}

\subjclass[2020]{11L03, 11A25, 20B35}
\keywords{derangements, affine group, roots of unity, cotangent power sums, Bernoulli polynomials}

\begin{document}

\begin{abstract}
Let $n$ be even, let $\zeta$ be a primitive $n$-th root of unity, and let
$\Delta_n$ be the set of derangements in the affine group
$\Aff(\Z/n\Z)$.  We evaluate the Wang--Sun type sum
\[
 S_n=\sum_{\pi\in\Delta_n}\sign(\pi)
 \prod_{x\in\Z/n\Z}
 \frac{1+\zeta^{x-\pi(x)}}{1-\zeta^{x-\pi(x)}}.
\]
The main step is an arithmetic grouping: for an affine map
$\pi(x)=vx+u$, the product depends only on
$\sigma=\gcd(1-v,n)$ and on the residue class of $u$ modulo $\sigma$.
After summing the translation signs, the full
sum becomes a sum over even divisors with explicit local multiplicities.  The only
remaining analytic blocks are alternating even cotangent power sums;
a recent formula of Liu and Xin then gives a closed expression in
Bernoulli polynomials and universal coefficients. Moreover, using
results by Cvijovi\'{c} and Klinowski we can evaluate the
cotangent power sums using exclusively rational arithmetic.  This supplies a
general formula for the full affine sum whose initial values were
computed using the definition only in an earlier work. The formula also lets us
prove some general patterns of the sum: it has integral value for powers of $2$,
prime denominator for $2p$ for primes $p>3$ and it has alternating signs.
\end{abstract}

\maketitle

\enlargethispage{2pt}

\section{Introduction}

Let $n>1$ be an odd integer, $D(n-1)$ be the set of all derangements of $\{1,\ldots,n-1\}$
and $\zeta$ a primitive $n$-th root of unity. Wang and Sun \cite{WangSun} proved the formula
\[
\sum_{\tau \in D(n-1)}\sign(\tau)\prod_{j=1}^{n-1}\frac{1+\zeta^{j-\tau(j)}}{1-\zeta^{j-\tau(j)}}=(-1)^{\frac{n-1}{2}}\frac{((n-2)!!)^{2}}{n}
\]
as a byproduct of their investigations of the determinants of certain trigonometric matrices. In
\cite{Agustin2023} the analogous sum over affine derangements
$G_{n}=\Aff(\Z/n\Z)$ was studied and a closed formula was obtained for the involutive
affine derangements. For the full set of affine derangements only
some values were calculated, with no conjectured closed formula.

The purpose of this note is to give such a formula and some of its consequences, with the following overall plan.
In Section \ref{Sec:Power} we examine how to represent the product of one term of the Wang--Sun sum as
the power of a cotangent term. Then, in Section \ref{Sec:Signs}, we determine the signs and the multiplicities of
their contributions using results by Lerch and Sita Ramaiah, respectively. This suffices for the calculations of Section \ref{Sec:CyclicandDihedral} of the Wang--Sun formula for the groups $\mathcal{T}_{n}$, the group of order $n$ of translations (which are transpositions in the musical context), and $\mathcal{D}_{n} = \{x\mapsto \pm x + u\}$, the dihedral group (which corresponds to inversions and transpositions in music). The full formula is presented in Section \ref{Sec:Formula}; this formula, together with a result
by Cvijovi\'{c} and Klinowski and Newton-Girard recursions examined in Section \ref{Sec:cotangent}, allows us to calculate it with rational numbers only. 
A closed version is given in Section \ref{Sec:Closed}, by a specialisation of a formula by Liu and Xin. Some
patterns evident from the calculations are proven in Section \ref{Sec:Patterns}, and then the calculations of \cite{Agustin2023}
are extended (as a means of providing concrete examples) up to $n=24$ in Section \ref{Sec:Computations}. Some final remarks are stated in Section \ref{Sec:Final}.

Throughout this article we use the notation of \cite{GKP}: $a\backslash b$ means $a$ divides $b$,
$a\perp b$ means $a$ and $b$ are coprime and $[P]$ is the Iverson bracket, which evaluates to $1$ when $P$ is true and $0$ otherwise.

\section{Affine derangements and power representation}\label{Sec:Power}

Throughout, $n\ge 2$ is even and $\zeta=\exp(2\pi i/n)$.  Put
\[
 G_n=\Aff(\Z/n\Z)
 =\{e^u.v:x\mapsto vx+u:\ u\in\Z/n\Z,\ v\in(\Z/n\Z)^\times\}.
\]

Let $\Delta_n\subseteq G_n$ denote the set of derangements.  We examine
\begin{equation}\label{eq:Sn}
 S_n=\sum_{\pi\in\Delta_n}\sign(\pi)
 \prod_{x=0}^{n-1}
 \frac{1+\zeta^{x-\pi(x)}}{1-\zeta^{x-\pi(x)}}.
\end{equation}

For a subset $A\subset G_{n}$, we define
\[
 S_{n}(A) = \sum_{\pi\in A\cap \Delta_{n}}\sign(\pi)
 \prod_{x=0}^{n-1}
 \frac{1+\zeta^{x-\pi(x)}}{1-\zeta^{x-\pi(x)}};
\]
in particular, $S_{n} = S_{n}(G_{n})$.

It is convenient to remark here that $S_{n}$ does not depend on the choice of the primitive root.
Indeed, suppose $\zeta' = \zeta^{a}$ with $a\perp n$ and define $h_{a}(x)=ax$. Then for a derangement $\pi=e^{u}.v
\in\Delta_{n}$ we have $\rho = h_{a}\circ \pi \circ h_{a}^{-1}$, and more explicitly
\[
\rho(x) = vx+au,
\]
which means that the effect of $h_{a}$ is a permutation of the translation parameters. Writing $y=ax$, we have
\[
 y-\rho(y) = y-vy-au=ax-vax-au = a(x-\pi(x))
\]
and therefore
\[
\prod_{x\in\Z/n\Z}\frac{1+(\zeta^{a})^{x-\pi(x)}}{1-(\zeta^{a})^{x-\pi(x)}} = \prod_{y\in \Z/n\Z}\frac{1+\zeta^{y-\rho(y)}}{1-\zeta^{y-\rho(y)}}.
\]

Since conjugation by $h_{a}$ permutes $\Delta_{n}$, preserves permutation signs, and leaves the linear parts invariant, summing this identity over $\pi\in\Delta_{n}$ proves the invariance. More generally, $S_{n}(A)$ is also independent of the choice of the primitive root whenever $A$ is closed under composition with translations.

Fix $\pi=e^u.v$. Set
\[
 \sigma=\gcd(1-v,n)=\gcd(v-1,n)\quad \text{and}\quad N=\frac {n}{\sigma},
\]
so $\sigma$ is even and $\pi$ has a fixed point if, and only if,
\[
 (1-v)x\equiv u\pmod n,
\]
which means that $\pi$ is a derangement if and only if $\sigma\centernot\backslash u$. Observe that the exponents of the
roots of unity in \eqref{eq:Sn} are
\[
 x-\pi(x) = x-(vx+u) = (1-v)x-u
\]
so we can find a relatively simple form of each summand, depending on $(1-v)x-u\equiv -u \equiv r \pmod{\sigma}$.

\begin{lemma}[Power representation]\label{lem:compression}
Choose $r$ with $r\equiv-u\pmod \sigma$.  Then
\begin{equation}\label{eq:compressed}
 \prod_{x=0}^{n-1}
 \frac{1+\zeta^{x-\pi(x)}}{1-\zeta^{x-\pi(x)}}
 =\Psi_{n,\sigma}(r)^\sigma,
\end{equation}
where, for a derangement,
\begin{equation}\label{eq:Psi}
 \Psi_{n,\sigma}(r)=
 \begin{cases}
 1,&N\text{ even},\\[4pt]
 \displaystyle i\cot\left(\frac{\pi r}{\sigma}\right),&N\text{ odd}.
 \end{cases}
\end{equation}
\end{lemma}

\begin{proof}
As $x$ runs through $\Z/n\Z$, the values
\[
 x-\pi(x)=(1-v)x-u
\]
run through the congruence class $r=-u\bmod \sigma$ in $\Z/n\Z$, each value occurring
exactly $\sigma$ times.  Hence
\[
 \prod_{x=0}^{n-1}
 \frac{1+\zeta^{x-\pi(x)}}{1-\zeta^{x-\pi(x)}}
 =\left(\prod_{m=0}^{N-1}
 \frac{1+\zeta^{r+\sigma m}}{1-\zeta^{r+\sigma m}}\right)^\sigma.
\]

Since $\zeta^\sigma$ is a primitive $N$th root of unity (for $\mathrm{ord}(\zeta^{\sigma})=\frac{n}{\gcd(n,\sigma)}=\frac{n}{\sigma} = N$),
then
\[
 z^{N}-1 = \prod_{m=0}^{N-1}(z-\zeta^{m\sigma}),
\]
and, upon substitutions $z=\zeta^{u}$ and $z=-\zeta^{u}$, we have
\[
 (\zeta^{u})^{N}-1 = \prod_{m=0}^{N-1}(\zeta^{u}-\zeta^{m\sigma})=\prod_{m=0}^{N-1}\zeta^{u}(1-\zeta^{-u+m\sigma})=(\zeta^{u})^{N}\prod_{m=0}^{N-1}(1-\zeta^{r+m\sigma})
\]
and
\[
 (-\zeta^{u})^{N}-1 = \prod_{m=0}^{N-1}(-\zeta^{u}-\zeta^{m\sigma})=\prod_{m=0}^{N-1}(-\zeta^{u})(1+\zeta^{-u+m\sigma})=(-\zeta^{u})^{N}\prod_{m=0}^{N-1}(1+\zeta^{r+m\sigma}).
\]

Therefore
\[
(-1)^{N}\frac{1-(-\zeta^{u})^N}{1-(\zeta^{u})^N}=
 \prod_{m=0}^{N-1}
 \frac{1+\zeta^{r+\sigma m}}{1-\zeta^{r+\sigma m}}.
\]

If $N$ is even this ratio is $1$.  If $N$ is odd and
$\omega=\exp(2\pi i/\sigma)$, then $(\zeta^{u})^N=\omega^{u}$ and
\[
 -\frac{1+\omega^{u}}{1-\omega^{u}}
 =-i\cot\left(\frac{\pi u}{\sigma}\right)=-i\cot\left(\frac{\pi(-r)}{\sigma}\right)=i\cot\left(\frac{\pi r}{\sigma}\right).
\]

When $\pi$ is a derangement, $r\not\equiv0\pmod \sigma$, and thus no
denominator vanishes.
\end{proof}

\section{Permutation signs and local multiplicities}\label{Sec:Signs}

Since the sign is a group homomorphism, we have
\[
 \sign(e^u.v)=\sign(e^u.1)\sign(e^0.v),
\]
so we can examine $e^{u}.1$ and $e^{0}.v$ separately.

\begin{lemma}[Translation signs]\label{lem:translation-sign}
For even $n$,
\[
 \sign(e^u.1)=(-1)^u.
\]

Consequently, for $\sigma=\gcd(v-1,n)$ and $1\leq r\leq \sigma-1$,
\begin{equation}\label{eq:residue-sign-sum}
 \sum_{\substack{0\le u<n\\u\equiv-r\, (\mathrm{mod}\,\sigma)}}(-1)^u
 =\frac{n}{\sigma}(-1)^r
\end{equation}
\end{lemma}

\begin{proof}
Translation by $u$ has $\gcd(u,n)$ cycles, each of length
$n/\gcd(u,n)$. Each one has sign $(-1)^{\frac{n}{\gcd(u,n)}-1}$ and therefore the sign
of $e^{u}.1$ is
\[
 ((-1)^{\frac{n}{\gcd(u,n)}-1})^{\gcd(u,n)}=(-1)^{n-\gcd(u,n)}.
\]

Since $n$ is even, and
$\gcd(u,n)$ and $u$ have the same parity, then
\[
 (-1)^{n-\gcd(u,n)}=(-1)^{n}(-1)^{\gcd(u,n)} = (-1)^u
\]

For the second assertion write $u=-r+j\sigma$, with
$1\leq j\leq n/\sigma$.  Since $\sigma$ is even then all signs are equal, and the total follows at once.
\end{proof}

For the affine part of the permutation we use the elementary sign formula
\begin{equation}\label{eq:multiplication-sign}
 \sign(e^0.v)= (-1)^{\frac{(v-1)(n/2-1)}{2}}
\end{equation}
due to Lerch \cite[p. 36]{Lerch1896}.

\begin{lemma}[Multiplication signs]\label{lem:refl} Let $n$ be even, $v\in(\Z/n\Z)^{\times}$ and $\sigma=\gcd(v-1,n)$. Then the multiplication
permutation $e^{0}.v:x\mapsto vx$ of $\Z/n\Z$ has sign
\[
 \epsilon_{n}(\sigma)=(-1)^{\frac{\sigma\cdot(n/2-1)}{2}}.
\]
\end{lemma}
\begin{proof}
When $n/2$ is odd then $\sign(e^{0}.v)$ is $1$ regardless of the value of $v$. When $n/2$ is
even, then it is $(-1)^{\frac{v-1}{2}}$. This is $-1$ if, and only if, $v\equiv 3\pmod{4}$, which
is equivalent to $v_{2}(v-1)=1$, which in turn is equivalent to $v_{2}(\sigma)=1$.
\end{proof}

For a divisor $\sigma \backslash n$, let
\[
 N_n(\sigma)=|\{v\in(\Z/n\Z)^\times:\gcd(1-v,n)=\sigma\}|.
\]

For instance, for $n=12$ and $\sigma = \gcd(5-1,12)=4$, we have
\[
 N_{12}(4) = |\{v\in(\Z/12\Z)^{\times}:\gcd(v-1,12)=4\}|=1.
\]

\begin{theorem}[V. Sita Ramaiah, 1978, \cite{Ramaia}] Let $f$ be an arbitrary arithmetic function. Then, for every $n\in\mathbb{N}$
\[
 M_{f}(n):=\sum_{v=1}^{n} [v\perp n]f(\gcd(v-1,n)) = \varphi(n)\sum_{d\backslash n}\frac{(\mu * f)(d)}{\varphi(d)}.
\]
\end{theorem}

The following local description is immediate from the Chinese
remainder theorem and Ramaiah's theorem.

\begin{proposition}[Local multiplicities]\label{prop:local-counts}
Let $n=\prod_p p^{a}$ and write $e_{p}=v_p(\sigma)$ for the multiplicity of $p$ in the
prime factorization of $\sigma$.  Then
\[
 N_n(\sigma)=\prod_{p^{a}\backslash\backslash n}N_{p^a}(p^{e_{p}}),
\]
where
\[
 N_{p^a}(p^{e})=
 \begin{cases}
 p^{a-1}(p-2),&e=0,\\
 p^{a-e-1}(p-1),&1\le e<a,\\
 1,&e=a.
 \end{cases}
\]
\end{proposition}

\begin{proof} Since the function $h(n)=\gcd(v-1,n)$ is multiplicative, the Chinese remainder theorem
allows us to study $N_{p^{a}}(p^{e})$ and then multiply the results. If we define $f(x)=[x=p^{e}]$ and
$e=v_{p}(\sigma)$ then $M_{f}(p^{a}) = N_{p^{a}}(p^{e})$ and
\begin{align*}
 N_{p^{a}}(p^{e}) &= \sum_{v=1}^{p^{a}}[v\perp p^{a}][\gcd(v-1,p^{a})=p^{e}]\\
 &=\varphi(p^{a})\sum_{d\backslash p^{a}}\frac{1}{\varphi(d)}\sum_{st=d}\mu(s)[t=p^{e}]\\
 &=\varphi(p^{a})\sum_{e\leq \ell\leq a}\frac{1}{\varphi(p^{\ell})}\mu(p^{\ell}/p^{e}).
\end{align*}

This simplifies to
\begin{align*}
N_{p^{a}}(p^{e}) &= \varphi(p^{a})\sum_{\ell=0}^{\min(1,a-e)}\frac{\mu(p^{\ell})}{\varphi(p^{e+\ell})}\\
&=\begin{cases}
\varphi(p^{a})\left(1-\frac{1}{\varphi(p)}\right), & e=0,\\
\varphi(p^{a})\left(\frac{1}{\varphi(p^{e})}-\frac{1}{\varphi(p^{e+1})}\right), & 1\leq e < a,\\
\varphi(p^{a})\cdot \frac{1}{\varphi(p^{a})}, & e=a,
\end{cases}\\
&=\begin{cases}
p^{a-1}(p-2), & e=0,\\
p^{a-e-1}(p-1), & 1\leq e < a,\\
1, & e=a.
\end{cases} \qedhere
\end{align*}
\end{proof}

\begin{remark} Note that $N_{n}(\sigma)$ and Proposition \ref{prop:local-counts} apply to every
positive integer $n$, including odd $n$, with $N_{1}(1)=1$. We will use this fact later.
\end{remark}

\section{The cyclic and dihedral subgroups}\label{Sec:CyclicandDihedral}

Before proceeding to the full affine group, we find the sum for two smaller subgroups.
In this section $n\geq 4$. Let
\[
 \mathcal{T}_{n}=\{e^{u}.1:u\in\Z/n\Z\}
\]
be the cyclic translation subgroup and let
\[
 \mathcal{D}_{n}=\{e^{u}.\pm 1:u\in\Z/n\Z\}
\]
be the dihedral subgroup of order $2n$. For even $n$, set
\[
 C_{n} = \sum_{r=1}^{n-1}(-1)^{r}\left(i\cot\frac{\pi r}{n}\right)^{n}.
\]

\begin{proposition}[Translations]
For even $n$
\[
S_{n}(\mathcal{T}_{n}) = C_{n}.
\]
\end{proposition}
\begin{proof}
The translations $e^{u}.1$ with $u\neq 0$ are precisely the translation derangements. Here $\sigma = n$ and $N=\frac{n}{n}=1$
in Lemma \ref{lem:compression}, so for $r\equiv -u \pmod{n}$ their product is
\[
\left(i\cot\frac{\pi r}{n}\right)^{n}.
\]

Since $n$ is even, $(-1)^{u}=(-1)^{r}$. Summing over $u=1,\ldots,n-1$ gives the result.
\end{proof}

\begin{theorem}[Dihedral derangements]\label{Thm:DihDer} For even $n\geq 4$,
\[
S_{n}(\mathcal{D}_{n}) = C_{n}+[4\backslash n]\frac{n}{2}.
\]
\end{theorem}
\begin{proof}
The contribution of translations is $C_{n}$ by the preceding proposition. It remains to consider the reflections $e^{u}.-1:x\mapsto -x+u$. Such a reflection has a fixed point if and only if
\[
 2x \equiv u \pmod{n}
\]
has a solution. Since $n$ is even, it is a derangement only when $u$ is odd. Thus there are $n/2$ reflective derangements.  For a reflection $\sigma = \gcd(-2,n) = 2$ and $N=n/2$. If $n\equiv 2 \pmod{4}$, then $N$ is odd, $r\equiv -u \equiv 1\pmod{2}$, and Lemma \ref{lem:compression} gives
\[
\left(i\cot\frac{\pi}{2}\right)^{2} = 0.
\]

Thus all reflective derangements contribute zero. If $4\backslash n$, then $N$ is even and their product is $1$. For $u$ odd, Lemma \ref{lem:translation-sign} gives $\sign(e^{u}.1)=-1$. Taking the representative $v=n-1$ for $-1$ in Lemma \ref{lem:refl} we have
\[
 \sign(e^{0}.-1) = (-1)^{\frac{n-2}{2}} = -1,
\]
for $4\backslash n$. Hence every derangement reflection has positive sign, and their total contribution is $n/2$.
\end{proof}

The involutive derangements of $\mathcal{D}_{n}$ have quite a simple sum. These are the dihedral quasipolarities, and we denote their set by $\mathcal{Q}_{n}^{\mathrm{dih}}$.

\begin{corollary}[Dihedral quasipolarities] For even $n\geq 4$.
\[
S_{n}(\mathcal{Q}_{n}^{\mathrm{dih}}) = [4\backslash n]\frac{n}{2}.
\]
\end{corollary}
\begin{proof}
The only translation derangement that is involutive is the antipodal one: $e^{n/2}.1$. Its product vanishes because $\zeta^{-n/2}=-1$, so every factor has numerator $1+\zeta^{-n/2}=0$. Thus only reflective derangements contribute, and we know that their contribution is $\frac{n}{2}$ when $4\backslash n$, as was seen in the proof of Theorem \ref{Thm:DihDer}.
\end{proof}

\begin{remark}
It is interesting to note that the sum for dihedral quasipolarities coincides with that for affine quasipolarities when $n/2$ is odd, and it retrieves a summand of the more general formula for $n/2$ even.
\end{remark}

\section{The divisor formula}\label{Sec:Formula}
We can now state the main reduction.

\begin{theorem}[Affine Wang--Sun reduction]\label{thm:reduction}
For even $n$,
\begin{equation}\label{eq:divisor-formula}
 S_n=\sum_{\substack{\sigma \backslash n\\\sigma\ \mathrm{even}}}
 \epsilon_n(\sigma)N_n(\sigma)\frac n\sigma\,J_{n,\sigma},
\end{equation}
where
\begin{equation}\label{eq:Jnd}
 J_{n,\sigma}=
 \begin{cases}
 -1,&n/\sigma\text{ even},\\[5pt]
 \displaystyle\sum_{r=1}^{\sigma-1}(-1)^r
 \left(i\cot\frac{\pi r}{\sigma}\right)^\sigma,&n/\sigma\text{ odd}.
 \end{cases}
\end{equation}
\end{theorem}

\begin{proof}
Group \eqref{eq:Sn} first by $\sigma=\gcd(1-v,n)$ and then by the nonzero
residue $r\equiv-u\pmod \sigma$.  Lemma~\ref{lem:compression} evaluates the
product.  Lemma~\ref{lem:translation-sign} shows that the translations in a fixed residue
class contribute $(n/\sigma)(-1)^r$.

The multiplication sign is constant among all $v$ satisfying
$\gcd(1-v,n)=\sigma$ and is equal to $\epsilon_n(\sigma)$, while the number of
such multipliers is $N_n(\sigma)$.  If $n/\sigma$ is even, then
$\Psi_{n,\sigma}(r)=1$ and, since $\sigma$ is even,
\[
 \sum_{r=1}^{\sigma-1}(-1)^r=-1.
\]
If $n/\sigma$ is odd, Lemma~\ref{lem:compression} gives the second line of
\eqref{eq:Jnd}.
\end{proof}

\begin{remark}\label{rem:2part}
If $n/\sigma$ is odd, then $\sigma$ contains the full $2$-primary part of $n$.
In particular $v_2(\sigma)=v_2(n)$, and hence
\[
 \epsilon_n(\sigma)=1
\]
throughout the cotangent case of \eqref{eq:Jnd}.  Thus the only
negative multiplication-sign correction occurs in the elementary
case $n/\sigma$ even.
\end{remark}

\section{Evaluation of the cotangent summation}\label{Sec:cotangent}

For even $\sigma=2m$, put
\[
 C_{\sigma}=\sum_{r=1}^{\sigma-1}(-1)^r
 \left(i\cot\frac{\pi r}{\sigma}\right)^\sigma.
\]

Splitting it by even and odd indices, we have
\begin{align*}
C_{\sigma} &= \sum_{r=1}^{\sigma-1}[2\backslash r]\left(i\cot\frac{\pi r}{\sigma}\right)^\sigma
-\sum_{r=1}^{\sigma-1}[2\centernot\backslash r]\left(i\cot\frac{\pi r}{\sigma}\right)^\sigma\\
&=\sum_{r=1}^{\sigma-1}[2\backslash r]i^{\sigma}\left(\cot\frac{\pi r}{\sigma}\right)^\sigma
-\sum_{r=1}^{\sigma-1}[2\centernot\backslash r]i^{\sigma}\left(\cot\frac{\pi r}{\sigma}\right)^\sigma\\
&=\sum_{r=1}^{2m-1}[2\backslash r](-1)^{m}\left(\cot\frac{\pi r}{2m}\right)^{2m}
-\sum_{r=1}^{2m-1}[2\centernot\backslash r](-1)^{m}\left(\cot\frac{\pi r}{2m}\right)^{2m}\\
&=\sum_{r=1}^{m-1}(-1)^{m}\left(\cot\frac{\pi 2r}{2m}\right)^{2m}
-\sum_{r=0}^{m-1}(-1)^{m}\left(\cot\frac{\pi (2r+1)}{2m}\right)^{2m}\\
&=\sum_{r=1}^{m-1}(-1)^{m}\left(\cot\frac{\pi 2r}{2m}\right)^{2m}
-\sum_{r=0}^{m-1}(-1)^{m}\left(\cot\frac{\pi (2r+1)}{2m}\right)^{2m}\\
&\phantom{=}+\sum_{r=1}^{m-1}(-1)^{m}\left(\cot\frac{\pi 2r}{2m}\right)^{2m}-\sum_{r=1}^{m-1}(-1)^{m}\left(\cot\frac{\pi 2r}{2m}\right)^{2m}\\
&=2\sum_{r=1}^{m-1}(-1)^{m}\left(\cot\frac{\pi r}{m}\right)^{2m}-\sum_{r=1}^{2m-1}(-1)^{m}\left(\cot\frac{\pi r}{2m}\right)^{2m}\\
&=(-1)^{m}\left[2\sum_{r=1}^{m-1}\left(\cot\frac{\pi r}{m}\right)^{2m}-\sum_{r=1}^{2m-1}\left(\cot\frac{\pi r}{2m}\right)^{2m}\right]
\end{align*}

If we define
\[
 P_{2m}(q) = \sum_{r=1}^{q-1}\left(\cot\frac{\pi r}{q}\right)^{2m}
\]
then we have
\[
 C_{\sigma} = (-1)^{m}(2P_{2m}(m)-P_{2m}(2m)).
\]

\begin{proposition}(Cvijovi\'{c}-Klinowski, 2000, \cite{CK00})
The polynomial
\[
p_{q}(x)=\sum_{r=0}^{\lfloor\frac{q-1}{2}\rfloor}(-1)^{r}\binom{q}{2r+1}x^{q-2r-1}
\]
has simple roots
\[
x_{r} = \cot\frac{\pi r}{q}, \quad 1\leq r \leq q-1.
\]
\end{proposition}

In particular, after dividing the coefficients of the polynomial $p_{q}(x)$ by its principal coefficient (which
is $\binom{q}{1}=q$) so it is monic, the resulting coefficients are
\[
\sigma_{j} = [2\backslash j](-1)^{j/2}\frac{1}{q}\binom{q}{j+1},
\]
with the convention that $\binom{m}{n}=0$ for $n>m$. By the Newton-Girard formulas (see \cite[Theorem 10.12.2]{Seroul}), we have
\begin{align*}
 P_{2m}(q) &= -2m\sigma_{2m}-\sum_{i=1}^{2m-1}(-1)^{i}\sigma_{2m-i}P_{i}(q)\\
 &=-2m\sigma_{2m}-\sum_{j=1}^{m-1}\sigma_{2m-2j}P_{2j}(q)
\end{align*}
where the $P_{i}$ with odd $i$ vanish. In particular
\begin{align*}
P_{2}(q) &= -2\sigma_{2} = -2(-1)^{1}\frac{1}{q}\binom{q}{3} = 2\frac{q(q-1)(q-2)}{3!q} = \frac{1}{3}q^{2}-q+\frac{2}{3},\\
P_{4}(q) &= -4\sigma_{4}-\sigma_{2}P_{2}(q) = -4(-1)^{2}\frac{1}{q}\binom{q}{5}-(-1)^{1}\frac{1}{q}\binom{q}{3}(\tfrac{1}{3}q^{2}-q+\tfrac{2}{3})\\
&=\frac{1}{45}q^{4}-\frac{4}{9}q^{2}+q-\frac{26}{45},
\end{align*}
etcetera. This allows us to calculate $C_{\sigma}$ with rational arithmetic alone.

\section{Closed expression}\label{Sec:Closed}

To obtain a closed expression, we use a result by Liu and Xin \cite[Corollary~4.12]{LiuXin}.
They define coefficients $r_{q,t}$ by
\[
 \left(\frac{e^s+1}{e^s-1}\right)^q
 =2^q\sum_{t\ge0}r_{q,t}s^{2t-q},
\]
so their alternating cotangent identity implies the following.

\begin{proposition}[Closed cotangent block]\label{prop:Cclosed}
For every even $\sigma$,
\begin{equation}\label{eq:Cclosed}
 C_\sigma=-2^\sigma\sum_{t=0}^{\sigma/2}
 r_{\sigma,t}\,
 \frac{B_{\sigma-2t}(1/2)}{(\sigma-2t)!}\,
 \sigma^{\sigma-2t},
\end{equation}
where $B_{j}(x)$ is the $j$-th Bernoulli polynomial, defined by
\[
 \frac{xe^{tx}}{e^{x}-1} = \sum_{j=0}^{\infty}B_{j}(t)\frac{x^{j}}{j!}.
\]
\end{proposition}

\begin{proof}
Put $\sigma=2m$.  Liu and Xin prove
\[
 \sum_{r=1}^{\sigma-1}(-1)^r\cot^{2m}\frac{\pi r}{\sigma}
 =(-1)^{m+1}2^{2m}\sum_{t=0}^{m}
 r_{2m,t}\frac{B_{2(m-t)}(1/2)}{(2(m-t))!}
 \sigma^{2(m-t)}.
\]
Multiplication by $i^{2m}=(-1)^m$ yields \eqref{eq:Cclosed}.
\end{proof}

Combining Theorem~\ref{thm:reduction} and
Proposition~\ref{prop:Cclosed} gives the promised explicit evaluation.

\begin{corollary}[Closed formula for the affine Wang--Sun sum]
For even $n$,
\begin{equation}\label{eq:final}
 S_n=-\sum_{\substack{\sigma\backslash n\\\sigma\ \mathrm{even}}}
 \epsilon_n(\sigma)N_n(\sigma)\frac n\sigma\left([2\backslash \tfrac{n}{\sigma}]+[2\centernot\backslash\tfrac{n}{\sigma}]2^\sigma\sum_{t=0}^{\sigma/2}
 r_{\sigma,t}\frac{B_{\sigma-2t}(1/2)}{(\sigma-2t)!}\sigma^{\sigma-2t}\right).
\end{equation}
\end{corollary}

\section{Some patterns of the sum}\label{Sec:Patterns}

\begin{corollary}[Integer values]
If $\ell\geq 1$, then $S_{2^{\ell}}$ is an integer.
\end{corollary}
\begin{proof}
The values of $\sigma_{j} = [2\backslash j](-1)^{j/2}\frac{1}{q}\binom{q}{j+1}$ are integers. Indeed, suppose $q$ is a power
of $2$ and $\sigma_{j}\neq 0$. Then $j+1$ is odd and the recursion
\[
(j+1)\binom{q}{j+1}=q\binom{q-1}{j},
\]
together with the fact that $(j+1)\perp q$, gives us $q\backslash \binom{q}{j+1}$. By the Newton-Girard recursion, $C_{2^{\ell}}$ is a sum of integers, which was the only possible non-integral summand, for $\frac{2^{\ell}}{2^{\ell}}=1$ is odd.
\end{proof}

\begin{theorem}[Prime denominators]
If $p>3$ is a prime, then $S_{2p}$ has denominator $p$.
\end{theorem}

\begin{proof}
The only two even divisors of $2p$ are $2$ and $2p$. For $\sigma=2$ we have $n/2 = p$ is odd and $C_{2}=0$, while
for $\sigma = 2p$ we have $N_{2p}(\sigma)=1$, $\epsilon_{2p}(\sigma)=1$ and $n/\sigma = 1$, and thus $S_{n}=C_{2p}$.

Let $\zeta=\exp(\pi i/p)$, $\omega = \zeta^{2}$,
\[
 \alpha_{j}=\frac{1+\omega^{j}}{1-\omega^{j}} = i\cot\frac{\pi j}{p}
\]
for $1\leq j\leq p-1$ and
\[
\beta_{r}=\frac{1+\zeta^{r}}{1-\zeta^{r}}= i\cot\frac{\pi r}{2p}
\]
for $r$ odd.

We can write
\[
C_{2p} = A_{p}-B_{p}
\]
where
\[
 A_{p} = \sum_{j=1}^{p-1}\alpha_{j}^{2p}\quad\text{and}\quad B_{p} = \sum_{r=1}^{2p-1}[2\centernot\backslash r]\beta^{2p}_{r}.
\]

The $\alpha_{j}$ are the roots of the polynomial
\[
 F_{p}(x)=\frac{(x+1)^{p}-(x-1)^{p}}{2p}
\]
which expands as
\[
F_{p}(x) = x^{p-1}+\frac{1}{p}\binom{p}{3}x^{p-3}+\frac{1}{p}\binom{p}{5}x^{p-5}+\cdots+\frac{1}{p}.
\]

Note that all coefficients are integers except the last one since $\frac{1}{p}\binom{p}{j}$ is an integer, for the
numerators of $\binom{p}{j}$ are divisible by $p$ while their denominators are not.

On the other hand, for $r$ odd, $(\zeta^{r})^{p} = -1$ and $\zeta^{r}=\frac{\beta_{r}-1}{\beta_{r}+1}$, hence
$(\beta_{r}-1)^{p} = -(\beta_{r}+1)^{p}$, and thus the $p$ numbers $\beta_{r}$ are the roots of
\[
 G_{p}(x) = \frac{(x+1)^{p}+(x-1)^{p}}{2}.
\]

Newton-Girard formulas now ensure $B_{p}$ is an integer, so we can focus on $A_{p}$. If $e_{i}$ is
the $i$-th elementary symmetric function on $m$ variables and $p_{r}=\sum_{j}x_{j}^{r}$ are the power sums, then
\[
 E(t) = \prod_{j=1}^{m}(1+x_{j}t) = \sum_{r=0}^{m}e_{r}t^{r}
\]
(see \cite[pp. 19--23]{Macdonald}) and, in particular,
\begin{align*}
\log E(t) &= \sum_{j=1}^{m}\log(1+x_{j}t)\\
&= \sum_{j=1}^{m}\sum_{r=1}^{\infty} (-1)^{r-1}x_{j}^{r}\frac{t^{r}}{r}\\
&= \sum_{r=1}^{\infty}(-1)^{r-1}\left(\sum_{j=1}^{m}x_{j}^{r}\right)\frac{t^{r}}{r}\\
&= \sum_{r=1}^{\infty}(-1)^{r-1}p_{r}\frac{t^{r}}{r}.
\end{align*}

From the expression for $F_{p}$ we have
\[
e_{j} = [2\backslash j]\frac{1}{p}\binom{p}{j+1},
\]
and, in particular,
\[
e_{2} = \frac{1}{p}\binom{p}{3}
\]
is an integer. We can write
\[
 E(t) = B(t) + \frac{t^{m}}{p} = B(t)\left(1+\frac{t^{m}}{pB(t)}\right),
\]
and note that $B(t) \in \mathbb{Z}[t]$ and $B(0) = 1$. Taking logarithms we get
\[
\log E(t) = \log B(t) + \log\left(1+\frac{t^{m}}{pB(t)}\right),
\]
and note that for $F_{p}$ we have $m=p-1$ roots, and thus for $p\geq 5$ we have
\[
 3m = 3(p-1) > 2p = 2m+2.
\]

Since
\[
\log\left(1+\frac{t^{m}}{pB(t)}\right) = \frac{t^{m}}{pB(t)}-\frac{t^{2m}}{2p^{2} (B(t))^{2}} +\cdots
\]
then the terms $\frac{t^{jm}}{jp^{j}B(t)^{j}}$ for $j\geq 3$ do not contribute to the coefficient of $t^{2p}$, since $3m>2p$. Therefore
\[
\frac{A_{p}}{-2p}=[t^{2p}]\log E = [t^{2p}]\log B + \frac{1}{p}[t^{m+2}]B^{-1}-\frac{1}{2p^{2}}[t^{2}]B^{-2}.
\]

A term of the form
\[
 r[t^{r}]\log B =  [t^{r-1}]\frac{B'(t)}{B(t)}
\]
is an integer since the derivative of a polynomial with integer coefficients has integer coefficients and the
formal inverse of $B(t)$ has integer coefficients because $B(0) = 1$. In particular, $2p[t^{2p}]\log(B)$ is an integer. From
this we also deduce that
\[
2p\frac{1}{p}[t^{m+2}]B^{-1} = 2[t^{m+2}]B^{-1}
\]
is an integer. For the last term, we have
\begin{align*}
B(t) &= 1+e_{2}t^{2}+O(t^{4}),\\
B(t)^{-1} &= 1-e_{2}t^{2} + O(t^{4}),\\
B(t)^{-2} &= 1-2e_{2}t^{2}+O(t^{4}),
\end{align*}
which implies that
\[
-\frac{1}{2p^{2}}[t^{2}]B^{-2} = \frac{e_{2}}{p^{2}}.
\]
and thus
\[
A_{p} = -2p[t^{2p}]\log E(t) \equiv -\frac{2e_{2}}{p} \pmod{\mathbb{Z}}.
\]

Finally
\[
 e_{2} = \frac{1}{p}\binom{p}{3} = \frac{(p-1)(p-2)}{6} \equiv \frac{1}{3}\pmod{p},
\]
which means that $pA_{p} \equiv -2e_{2} \equiv -\frac{2}{3}\pmod{p}$. Since this is non-zero,
then the reduced denominator of $A_{p}$ is exactly $p$. Consequently the reduced denominator of $C_{2p}$
is exactly $p$.
\end{proof}

The alternating sign pattern of the sum that is evident from the first terms is also prevalent.

\begin{lemma}\label{L:Signo}
For $m\geq 2$
\[
\sign(C_{2m}) = (-1)^{m+1}
\]
and $C_{2}=0$.
\end{lemma}
\begin{proof}
Since $i^{2m}=(-1)^{m}$, we can pair the terms $r$ and $2m-r$ in $C_{2m}$. The middle term $r=m$
vanishes because $\cot(\pi/2)=0$. Therefore
\[
 C_{2m} = 2(-1)^{m}\sum_{r=1}^{m-1}(-1)^{r}\left(\cot\frac{\pi r}{2m}\right)^{2m}
\]
and the terms
\[
a_{r} = \left(\cot\frac{\pi r}{2m}\right)^{2m}
\]
are strictly decreasing for $1\leq r \leq m-1$ because $\cot(x)$ is strictly decreasing on $(0,\pi/2)$. By
the properties of alternating sums, its sign is negative, and multiplying by $(-1)^{m}$ gives the result.
\end{proof}

\begin{theorem}[Alternating signs] For every even integer $n\geq 6$
\[
\sign(S_{n}) = (-1)^{n/2+1}.
\]
\end{theorem}
\begin{proof}
When $n=2q$ with $q$ odd, every even divisor is of the form $\sigma = 2f$ with $f\backslash q$, which
means that $n/\sigma = q/f$ is odd. Hence only cotangent sums contribute to $S_{n}$. Furthermore,
$\epsilon_{n}(\sigma) = (-1)^{\frac{\sigma\cdot (n/2-1)}{2}}= 1$ since $n/2 = q$ is odd. By Lemma \ref{L:Signo}
\[
\sign(C_{2f})=(-1)^{f+1} = 1
\]
since $f>1$ is odd. For the $f=1$ term we have $C_{2}=0$, and thus all summands of $S_{n}$ are non-negative
when $n=2q\geq 6$. Since $N_{2q}(2q)=1$ and $C_{2q}>0$, it follows that $S_{2q}>0$.

When $n=2^{a}q$ with $a\geq 2$ and $q$ odd, the cotangent summands correspond to divisors of the form
$\sigma =2^{a}f$ with $f\backslash q$, because $n/\sigma$ must be odd. But then by Lemma \ref{L:Signo}
\[
 \sign(C_{\sigma}) = (-1)^{2^{a-1}f+1} = -1.
\]
and $\epsilon_{n}(\sigma) = 1$ because $\sigma/2$ is even.

The possible positive contributions then come from the cases $\sigma = 2^{e}f$ with $1\leq e < a$, $f\backslash q$.
Put
\[
 \mathcal{A}(q) = \sum_{f\backslash q} N_{q}(f)\frac{q}{f},
\]
which is positive. For the factors related to the prime $2$, note first that $n/\sigma$ is even, thus the contribution of $J_{n,\sigma}$ is a $-1$ factor. We first examine $e=1$, which contributes the positive term (since $\epsilon_{n}(\sigma)=-1$)
\[
N_{2^{a}}(2)\frac{2^{a}}{2^{1}}=2^{a-2}2^{a-1}=2^{2a-3};
\]
we follow with $2\leq e < a$, which contribute the negative terms (since $\epsilon_{n}(\sigma)=1$),
\[
N_{2^{a}}(2^{e})\frac{2^{a}}{2^{e}}=2^{a-e-1}2^{a-e} = 2^{2a-2e-1}.
\]

Thus the total contribution is
\begin{align*}
E_{n} &= \left(2^{2a-3}-\sum_{e=2}^{a-1}2^{2a-2e-1}\right)\mathcal{A}(q)\\
&=\left(2^{2a-3}-\sum_{e=0}^{a-3}2^{2a-2e-5}\right)\mathcal{A}(q)\\
&=\frac{2^{2a-2}+2}{2^{2}-1}\mathcal{A}(q)=\frac{2^{2a-2}+2}{3}\mathcal{A}(q)>0,
\end{align*}
and it is bounded above as follows: first $\sum_{f\backslash q} N_{q}(f)=\phi(q)$, and thus
\[
\mathcal{A}(q)\leq q \sum_{f\backslash q} N_{q}(f) = q \phi(q) \leq q^{2}.
\]

Since
\[
 \frac{2^{2a-2}+2}{3}\leq 2^{2a-2}
\]
for $a\geq 2$, then
\[
E_{n} \leq 2^{2a-2}q^{2} = \frac{(2^{a}q)^{2}}{2^{2}} = \frac{n^{2}}{4}.
\]

In order to bound $|C_{n}|$ from below for $n=2^{a}q\geq 8$, we first define
\[
a_{r} = \left(\cot\frac{r\pi}{n}\right)^{n}
\]
for $1\leq r\leq n/2-1$. Since $n/2$ is even, we can write
\[
|C_{n}|=2\sum_{r=1}^{n/2-1}(-1)^{r+1} a_{r} \geq 2(a_{1}-a_{2}).
\]

By the double angle formula for tangents, we have for $0<x<\pi/4$ that
\[
\frac{\cot 2x }{\cot x } = \frac{1-(\tan(x))^{2}}{2}< \frac{1}{2}
\]
and thus
\[
\frac{a_{2}}{a_{1}} = \frac{\left(\cot\frac{2\pi}{n}\right)^{n}}{\left(\cot\frac{\pi}{n}\right)^{n}}<\left(\frac{1}{2}\right)^{n}.
\]

Since the cotangent is decreasing, then for $n\geq 8$
\[
\cot\frac{\pi}{n} \geq \cot \frac{\pi}{8} = 1+\sqrt{2} >2
\]
which means that $a_{1}=(\cot\frac{\pi}{n})^{n}>2^{n}$. Hence,
\[
 |C_{n}| \geq 2(a_{1}-a_{2}) \geq 2\left(a_{1}-\frac{a_{1}}{2^{n}}\right) > 2\cdot 2^{n}\left(1-\frac{1}{2^{n}}\right)
 =2^{n+1}-2> \frac{n^{2}}{4}.
\]

Given that the absolute value of $C_{n}$ is larger than the upper bound on the largest positive contribution, and
other cotangent terms are also negative, it follows that $S_{n}<0$ for $n=2^{a}q$ with $a\geq 2$ and $q$ odd.
\end{proof}

\section{Comparison with previously computed values}\label{Sec:Computations}

Formula \eqref{eq:final} reproduces the values computed in
\cite{Agustin2023}.  We use the modulus $n$ as the first column, as in
the displayed numerical table of that paper.
\begin{table}[ht!]
\caption{The values of $S_{n}$ for selected even indices $n$.}
\begin{tabular}{c|r}
\hline
 $n$&$S_n$\\
\hline
$6$&$1456/27$\\
$8$&$-2300$\\
$10$&$762256/5$\\
$12$&$-10643506432/729$\\
$14$&$13444304416/7$\\
$16$&$-332995177452$\\
$18$&$1450048309488389824/19683$\\
$20$&$-63332915694394741504/3125$\\
$22$&$74674004155100687824/11$\\
$24$&$-1446060869245578144968989936/531441$\\
\hline
\end{tabular}
\end{table}

For example, when $n=8$, the contributions of the even divisors are
\[
 \sigma=2:8,\qquad \sigma=4:-2,\qquad \sigma=8:-2306,
\]
and hence
\[
 S_8=8-2-2306=-2300.
\]

The large size of the numbers is a consequence of the fact that the dominant
term is often the cotangent block with $\sigma=n$, while proper even divisors
supply arithmetic correction terms.

We now illustrate the formula for $n=20,22,24$. For $n=20$ we have the
even values $\sigma = 2,4,10,20$, and the corresponding signs are
\begin{align*}
\epsilon_{20}(2)&=(-1)^{\frac{2\cdot(10-1)}{2}}=-1,\\
\epsilon_{20}(4)&=(-1)^{\frac{4\cdot(10-1)}{2}}=1,\\
\epsilon_{20}(10)&=(-1)^{\frac{10\cdot(10-1)}{2}}=-1,\\
\epsilon_{20}(20)&=(-1)^{\frac{20\cdot(10-1)}{2}}=1
\end{align*}
with multiplicities
\begin{align*}
N_{20}(2) &= N_{2^{2}}(2)N_{5}(1) = 2^{2-1-1}5^{0}(5-2) = 3,\\
N_{20}(4) &= N_{2^{2}}(4)N_{5}(1) = 1\cdot 5^{0}(5-2) = 3,\\
N_{20}(10) &= N_{2^{2}}(2)N_{5}(5) = 2^{2-1-1}\cdot 1 = 1,\\
N_{20}(20) &= N_{2^{2}}(4)N_{5}(5) = 1\cdot 1 = 1.
\end{align*}

Since $20/2 = 10$, $20/4 = 5$, $20/10 = 2$ and $20/20=1$, only $\sigma = 4,20$ contribute
a cotangent sum. We have
\begin{align*}
C_{4} &= (-1)^{2}(2P_{4}(2)-P_{4}(4)) = -2,\\
C_{20} & = (-1)^{10}(2P_{20}(10)-P_{20}(20)) = -\frac{63332915694394747754}{3125}.
\end{align*}

Thus
\begin{align*}
S_{20} &= (-1)(3)(10)(-1)+(1)(3)(5)(-2)+(-1)(1)(2)(-1)\\
&\phantom{=}+(1)(1)(1)(-\tfrac{63332915694394747754}{3125})\\
&=-\frac{63332915694394741504}{3125}.
\end{align*}

For $n=22$ the computation is particularly simple since the only nonzero term corresponds to $\sigma = 22$, and thus
\[
S_{22}=(-1)^{11}(2P_{22}(11)-P_{22}(22))= \frac{74674004155100687824}{11}.
\]

For $n=24$ we have the even values $\sigma = 2,4,8,6,12,24$, and the corresponding signs are
\begin{align*}
\epsilon_{24}(2)&=(-1)^{\frac{2\cdot(12-1)}{2}}=-1,\\
\epsilon_{24}(4)&=(-1)^{\frac{4\cdot(12-1)}{2}}=1,\\
\epsilon_{24}(6)&=(-1)^{\frac{6\cdot(12-1)}{2}}=-1,\\
\epsilon_{24}(8)&=(-1)^{\frac{8\cdot(12-1)}{2}}=1,\\
\epsilon_{24}(12)&=(-1)^{\frac{12\cdot(12-1)}{2}}=1,\\
\epsilon_{24}(24)&=(-1)^{\frac{24\cdot(12-1)}{2}}=1
\end{align*}
with multiplicities
\begin{align*}
N_{24}(2) &= N_{2^{3}}(2)N_{3}(1) = 2^{3-1-1}3^{0}(3-2) = 2,\\
N_{24}(4) &= N_{2^{3}}(4)N_{3}(1) = 2^{3-2-1}\cdot 3^{0}(3-2) = 1,\\
N_{24}(6) &= N_{2^{3}}(2)N_{3}(3) = 2^{3-1-1}\cdot 1 = 2,\\
N_{24}(8) &= N_{2^{3}}(8)N_{3}(1) = 1\cdot 3^{0}(3-2) = 1,\\
N_{24}(12) &= N_{2^{3}}(4)N_{3}(3) = 2^{3-2-1}\cdot 1 = 1,\\
N_{24}(24) &= N_{2^{3}}(8)N_{3}(3) = 1\cdot 1 = 1.
\end{align*}

Since $24/2 = 12$, $24/4 = 6$, $24/6 = 4$, $24/8=3$,  $24/12=2$ and $24/24=1$, only $\sigma = 8,24$ contribute
a cotangent sum. We have
\begin{align*}
C_{8} &= (-1)^{4}(2P_{8}(4)-P_{8}(8)) = -2306,\\
C_{24} & = (-1)^{12}(2P_{24}(12)-P_{24}(24)) = -\frac{1446060869245578141305235682}{531441}.
\end{align*}

Thus
\begin{align*}
S_{24} &= (-1)(2)(12)(-1)+(1)(1)(6)(-1)+(-1)(2)(4)(-1)\\
&\phantom{=}+(1)(1)(3)(-2306)+(1)(1)(2)(-1)+(1)(1)(1)(-\tfrac{1446060869245578141305235682}{531441})\\
&=-\frac{1446060869245578144968989936}{531441}.
\end{align*}

\section{Final remarks}\label{Sec:Final}

The original affine problem has two distinct layers.  The first is
finite-group arithmetic: the fixed-point equation, the distribution of
$\gcd(1-v,n)$, and the signs of affine permutations.  The second consists of
finite trigonometric power sums.  Theorem~\ref{thm:reduction} separates
these layers completely.

This separation also clarifies the role of the recent work
\cite{LiuXin}. What is specific to the
present problem is that the full signed sum over affine derangements
collapses to a divisor sum with the local factors $N_n(\sigma)$ and
$\epsilon_n(\sigma)$, and that the required cotangent sums for it can
be carried out with rational arithmetic only. Thus the direct and limited calculation of
\cite{Agustin2023} is replaced by a general arithmetic formula.

It may be worthwhile to study whether similar simplifications occur for
Wang--Sun type sums over other structured permutation groups, where the
difference map $x\mapsto x-\pi(x)$ has controlled fibres.

\section*{Declaration of generative AI and AI-assisted technologies in the manuscript preparation process}

During the preparation of this paper, the author used ChatGPT 5.6 Sol to investigate the open problem concerning the full affine Wang–Sun sum posed in his earlier work. The model proposed an initial formula and its evaluation through a specialization of Liu–Xin’s result. The form of the formula led the author to investigate the prime-denominator and alternating-sign patterns discussed in Section \ref{Sec:Patterns}. In response to the author’s questions, the model provided sketches of the corresponding arguments. The model also suggested a rational computation upon prompting and assisted in locating relevant literature. The author subsequently studied the sources, made the rational computation explicit through the Newton–Girard recurrence, and checked the resulting numerical evaluations with Maxima. The result by Lerch concerning the sign of multiplication by a unit was reported from third-hand sources, and Lerch's original formulation gave a cleaner statement. After independently investigating the connection with Menon-type identities (via \cite{Toth}), the author reconstructed the local-multiplicity argument using Sita Ramaiah's result. The author substantially rewrote and checked the proofs, particularly the prime-denominator and alternating-sign arguments, and wrote the final manuscript. The author takes full responsibility for its mathematical claims, computations, references, and exposition.


\begin{thebibliography}{99}

\bibitem{Agustin2023}
O.~A. Agust\'in-Aquino,
\emph{Wang--Sun formula in $\overrightarrow{\mathrm{GL}}(\mathbb Z/2k\mathbb Z)$},
Integers \textbf{23} (2023), Paper A37.

\bibitem{CK00}
D.~Cvijovi\'{c} and J. Klinowski,
\emph{Finite Cotangent Sums and the Riemann Zeta Function},
Math. Slovaca \textbf{50} (2000), No. 2, 149--157.

\bibitem{GKP}
R.~L. Graham, D.~E. Knuth, and O.~Patashnik,
\emph{Concrete Mathematics: A Foundation for Computer Science},
2nd ed.,
Addison--Wesley, Reading, MA, 1994.

\bibitem{Lerch1896}
M.~Lerch,
\emph{Sur un th\'eor\`eme arithm\'etique de Zolotarev},
Bulletin int. de l'Ac. Prague \textbf{3} (1896), 34--37.

\bibitem{LiuXin}
J.~Liu and G.~Xin,
\emph{Root-of-unity weighted trigonometric power sums: a constant term approach},
arXiv:2607.29130 [math.CO], 2026.

\bibitem{Macdonald}
I.~G. Macdonald,
\emph{Symmetric Functions and Hall Polynomials},
2nd ed.,
Oxford University Press, 1995.

\bibitem{Ramaia}
V.~Sita Ramaiah,
\emph{Arithmetical sums in regular convolutions},
Journal für die reine und angewandte Mathematik \textbf{303/304} (1978), 265--283.

\bibitem{Seroul}
R.~S\'eroul,
\emph{Programming for Mathematicians},
translated from the French by D.~O'Shea,
Universitext, Springer-Verlag, Berlin, 2000.

\bibitem{Toth}
L.~Tóth,
\emph{Proofs, generalizations and analogs of Menon's identity: a survey},
Acta Univ. Sapientiae, Mathematica \textbf{15}, No. 1, (2023), 142--197.

\bibitem{WangSun}
H.~Wang and Z.-W.~Sun,
\emph{Proof of a conjecture involving derangements and roots of unity},
Electron. J. Combin. \textbf{30} (2023), No.~2, Paper P2.1.

\end{thebibliography}
\end{document}